\documentclass{birkjour}
\usepackage[all]{xy}
\usepackage{fixltx2e}

 \newtheorem{thm}{Theorem}[section]
 
 \newtheorem{corollary}[thm]{Corollary}
 \newtheorem{lemma}[thm]{Lemma}
 \newtheorem{proposition}[thm]{Proposition}
 \theoremstyle{definition}
 \newtheorem{definition}[thm]{Definition}
 \theoremstyle{remark}
 
 \newtheorem*{example}{Example}
 \numberwithin{equation}{section}

\newcommand{\C}{\mathbb{C}}

\begin{document}

%
%
%
%
%
%
%
%
%

\title[Internal monoids and groups in ccm-magmas]{Internal monoids and groups in the category of commutative cancellative medial magmas}

\author[J. P. Fatelo]{Jorge Pereira Fatelo}

\address{%
School of Technology and Management\\
Centre for Rapid and Sustainable Product Development-CDRSP\\
Polytechnic Institute of Leiria\\
P-2411-901, Leiria, Portugal}

\email{jorge.fatelo@ipleiria.pt}

\author[N. Martins-Ferreira]{Nelson Martins-Ferreira}

\address{%
School of Technology and Management\\
Centre for Rapid and Sustainable Product Development-CDRSP\\
Polytechnic Institute of Leiria\\
P-2411-901, Leiria, Portugal}

\email{martins.ferreira@ipleiria.pt}

\thanks{ Research supported by
IPLeiria/ESTG-CDRSP and Funda\c c\~ao para a Ci\^encia e a
Tecnologia 
by the FCT projects PTDC/EME-CRO/120585/2010 and
PTDC/MAT/120222/2010. And also by the grant number SFRH/BPD/4321/2008 at CMUC}
\subjclass{Primary 08C15; Secondary 20N02}

\keywords{midpoint algebra,  commutative monoid, abelian group, cancellation law, medial law, quasigroup, internal monoid, internal group, internal relation, weakly Mal'tsev category}

\date{17Jun2014}
\dedicatory{\date}

\begin{abstract}
This article considers the category of commutative medial magmas with cancellation, a structure that generalizes midpoint algebras and commutative semigroups with cancellation. In this category each object admits at most one internal monoid structure for any given unit. Conditions for the existence of internal monoids and internal groups, as well as conditions under which an internal reflexive relation is a congruence,  are studied.
\end{abstract}

\maketitle

\section{Introduction}\label{sec introduction}

A midpoint algebra is a pair $(A,\oplus)$ where $A$ is a  set and $\oplus$ a binary operation satisfying the following axioms:
\[\begin{array}{rl}
 \text{(idempotency)} & x\oplus x= x,\\
 \text{(commutativity)} & x\oplus y= y \oplus x ,\\
 \text{(cancellation)} & (\exists a\in A, x\oplus a= y \oplus a)\Rightarrow x=y,\\
\text{(mediality)}&(x\oplus y)\oplus (z\oplus w)=( x \oplus z)\oplus (y\oplus w) .
\end{array}\]
The main example, which attracted our attention to this study, is the unit interval $A=[0,1]$ with $x\oplus y=\frac{x+y}{2}$. 

 The category of midpoint algebras is not a Mal'tsev category. Indeed, the usual order relation on the unit interval, with the arithmetic mean as above, is clearly a subalgebra of the product $[0,1]\times [0,1]$, which is reflexive and transitive but not symmetric, contradicting the well known characterization of Mal'tsev categories \cite{CKP,Carboni-Lambek-Pedicchio,Carboni-Pedicchio-Pirovano}. Nevertheless, the category of midpoint algebras does have some interesting properties weaker than those of Mal'tsev categories, namely that each object admits at most one internal monoid structure for each choice of a unit. To illustrate this aspect we use another example. Let $(A,\oplus)$ be the midpoint algebra with $A=]0,1]$, the set of positive real numbers smaller or equal than one, and the operation\[a\oplus b=\frac{2ab}{a+b}.\] It is clear that if $a$ and $b$ are positive real numbers then $a\oplus b$ is also positive; it is also easy to see that if $a$ and $b$ are less or equal than one then $a\oplus b$ is less or equal than one. A simple way to see it is to observe that the condition $a\oplus b\leq 1$ is equivalent to the condition\[0\leq a(1-b)+b(1-a).\] Hence the operation is well defined and it is easily checked that $(A,\oplus)$ is a midpoint algebra. 
%
The diagram\[\xymatrix{A\times A\ar[r]^{m}&A&\{1\}\ar[l]},\] with \[m(x,y)=\frac{xy}{x+y-xy},\] is an internal monoid. This internal monoid structure is uniquely determined by the unit element $1\in A$ and there is no internal monoid structure for any other choice of a unit element. Moreover, it is not an internal group as, for example, $\frac{1}{2}$ has no inverse since the equation \[x=1+x\] has no solution in the real numbers. 

%
%
%

While proving some of the results presented in this paper, we have observed that, in many cases, the idempotency law could be dealt with in a separate way. This suggested to us the study of the more general structure of commutative cancellative  medial magmas, which, for simplicity, we will refer to as ccm-magmas. The medial law has been widely studied over the last three quarter of a century. In the context of magmas (formerly known as groupoids) this law  was also referred to as bisymmetric,  entropic or transposition, among other names. It has been studied either from an algebraic point of view or from a more geometrical perspective, see for example \cite{Aczel,Escardo,Jezek,Kannappan,Kermit70,Toyoda}.

 The open-closed unit interval $]0,1]$ used in the example above has several other important structures of ccm-magma: that of a midpoint algebra (with the arithmetic mean), and that of a commutative monoid with cancellation (with the usual multiplication), resulting in an interplay between different ccm-magmas on the same set. The study of such interactions is certainly  worthy. Nevertheless, for the moment, we dedicate our attention to some important aspects of the internal structures in the category of ccm-magmas, such as internal monoid, internal group and internal relation.

This paper is organized as follows: Section \ref{sec preliminaries}  recalls some categorical notations and concepts such as the one of an internal monoid; Section \ref{sec ccm-magmas} introduces the category of ccm-magmas, gives some examples, and some useful lemmas are proven; Section  \ref{sec mal'tsev} observes that the category of ccm-magmas is weakly Mal'tsev and characterizes the existence of an internal monoid structure on a given object, for every choice of a unit element; Section \ref{sec monoids} explores this property further by considering the notions of $e$-expansive, $e$-symmetric and homogeneous ccm-magmas (a short list of simple examples and counter-examples is also presented at the end of the section); and finally, Section \ref{sec relations} studies some properties of internal relations, namely symmetry, transitivity, reflexivity and difunctionality, which are not visible in the case of Mal'tsev categories.

\section{Preliminaries}\label{sec preliminaries}

The basic notions and notations from Category Theory used in this paper can be found, for instance, in \cite{MacLane}. Let us just recall a few of them. If $\C$ is any category with finite limits, an internal monoid in $\C$ is a diagram of the shape\[\xymatrix{A\times A\ar[r]^{m}&A&1\ar[l]_{e}}\] in which $A$ is any object in $\C$, $m$ and $e$ are morphisms in $\C$, $1$ denotes the terminal object and  the following diagram is commutative (where $!_A$ denotes the unique morphism from $A$ to the terminal object $1$).\[\xymatrix{A^3\ar[r]^{1_A\times m}\ar[d]_{m\times 1_A}&A^2\ar[d]_{m}&A\ar@{=}[ld]\ar[l]_{\langle e!_A,1_A\rangle}\ar[d]^{\langle 1_A,e!_A\rangle}\\A^2\ar[r]^{m}&A&A^2\ar[l]_{m}}\] 
An internal monoid $(A,m,e)$ is an internal group (see for example the Appendix in \cite{BB}) if and only if the diagram \[\xymatrix{A&A\times A\ar[r]^{\pi_2}\ar[l]_{m}& A}\] is a product diagram (the morphism $\pi_2$ is the canonical second projection), in other words, for every two morphisms $u,v\colon{X\to A}$ there exists a unique morphism, represented as $u-v$, from $X$ to $A$ such that $$m\langle u-v,v \rangle=u.$$ The inverse of a generalized element $x\colon{X\to A}$ is clearly $e-x$ with $e\colon{1\to A}$ the unit element. This is simply another way to say that there is a morphism $t\colon{A\to A}$ such that \[m\langle 1_A,t\rangle=e!_A=m\langle t,1_A\rangle.\]

This paper restricts itself to quasi-varieties of universal algebra \cite{Sankappanavar}, that is, categories in which the objects are sets equipped with an arbitrary family of finite-arity operations, satisfying a collection of axioms which may be either expressed  as identities or as implications. All the results to be proven about conditions of uniqueness are easily generalized to a category with a faithful functor into the category of sets, preserving finite limits. However, the results involving existence conditions depend on the context and do not necessarily hold in general.

The notion of internal ccm-magma is also a natural one to be considered. It would be interesting, for example, to study topological ccm-magmas, i.e. ccm-magmas internal to the category of topological spaces  \cite{Escardo,Kermit68}. 

A congruence on an object $A$ is a subalgebra of $A\times A$, which, if considered as a relation, is reflexive, transitive and symmetric (see e.g. \cite{BB}).  We will also consider difunctional relations: a relation $R\subseteq X\times Y$ is difunctional if the implication \[xRy,zRy,zRw\Rightarrow xRw\]holds for all $x,z\in X$ and  $y,w\in Y$. 

A Mal'tsev category  is characterized by the property that every reflexive internal relation is a congruence relation, or equivalently, by the property that every internal relation is difunctional (\cite{CKP,Carboni-Lambek-Pedicchio,Carboni-Pedicchio-Pirovano} see also \cite{BB}). As we will see, this property does not hold in the categories we are considering. However, these have a weaker property: some types of  reflexive internal relations are still congruences, and some types of internal relations are automatically difunctional.

\section{Ccm-magmas}\label{sec ccm-magmas}

A ccm-magma (commutative cancellative medial magma) may be obtained from a midpoint algebra  simply by not requiring the idempotency axiom.

\begin{definition}\label{middle algebra}
A \emph{ccm-magma} is an algebraic structure $(A,\oplus)$ with a binary operation $\oplus$ satisfying the following axioms:
\begin{enumerate}
\item[M1] $a\oplus b=b\oplus a$
\item[M2] $a\oplus c=b\oplus c\implies a=b$
\item[M3] $(a\oplus b)\oplus (c\oplus d)=(a\oplus c)\oplus (b\oplus d)$
\end{enumerate}
\end{definition}

A morphism of ccm-magmas is simply a homomorphism, that is, a map preserving the binary operation.


It is known that the category of commutative cancellative magmas is weakly Mal'tsev (\cite{NMF2012}). The third axiom has an important consequence illustrated in the following proposition.

\begin{proposition}\label{prop1} If $f_1,f_2\colon{A\to B}$ are two morphisms of ccm-magmas, then the map $f\colon{A\times A\to B}$ defined by \[f(x,y)=f_1(x)\oplus f_2(y)\] is also a morphism of ccm-magmas.
\end{proposition}
\begin{proof}
The result is an immediate consequence of the axiom (M3).
\end{proof}

The result of Proposition \ref{prop1} stands even if commutativity is not included (see also \cite{Stein}, Theorem 12.8); in this case, among all medial-like identities studied in \cite{Polonijo}, (M3) is the most appropriate in proving this result. In fact, as we will see in Section \ref{sec mal'tsev}, even the weakly Mal'tsev property may be proven without commutativity. This paper considers the commutativity axiom. In a future work we will investigate a more general structure where (M1) is not included.

A simple variation of the previous result shows that for any two ccm-magmas $X$ and $Y$, there is a canonical morphism from product into co-product \[X\times Y\to X+Y\] sending each pair $(x,y)\in X\times Y$ to the element $\iota_X(x)\oplus \iota_Y(y)\in X+Y$, with $\iota_X\colon{X\to X+Y}$ and $\iota_Y\colon{Y\to X+Y}$ the canonical inclusions into the co-product. More generally, for any two given morphisms $f_1\colon{X\to B}$ and $f_2\colon{Y\to B}$ there is an induced morphism, say $f=f_1\oplus f_2\colon{X\times Y\to B}$, even thought there are no canonical injections to insert $X$ and $Y$ in $X\times Y$.

Next we present a short list of examples build up from well-known algebraic structures.

\begin{example}\label{ex1} A short list of simple examples:
\begin{enumerate}
\item The open unit interval $]0,1[$ with the operation \[x\oplus y=\frac{x+y}{2}\quad\text{or}\quad x\oplus y=\sqrt{xy}\] is a ccm-magma.
\item Every midpoint algebra is a ccm-magma.
\item The set of natural numbers with the usual addition is a ccm-magma. 
\item Every commutative semigroup with cancellation  is a ccm-magma.
\item If $A$ is a field, then the formula\[x\oplus y=a(x+y)+b\] gives a ccm-magma for every choice of $a,b\in A$ with $a\neq 0$.
\item If $A$ is any ring, then the formula\[x\oplus y=a(x+y)+b\] gives a ccm-magma for every choice of $a,b\in A$, with $a$ an invertible element.
\item If $V$ is any real or complex vector space, then the formula \[x\oplus y=\alpha(x+y)+b\]gives a ccm-magma on $V$ for every choice of a scalar $\alpha\neq 0$ and for every $b\in A$.
\item If $R$ is a ring and $A$ an $R$-module, then the formula \[x\oplus y=\alpha(x+y)+b\]gives a ccm-magma on $A$ for every choice of an invertible scalar $\alpha\in R$ and any $b\in A$.
\item If $A$ is  the set of positive real numbers, then the formula\[x\oplus y=\frac{axy}{x+y}\]gives a ccm-magma for every choice of $a\in A$.
\item If $(A,+,\times,\cdot,0,1)$ is a commutative and associative classical algebra over a commutative ring $R$, then the formula \[x\oplus y=\alpha\cdot(x+y)+\beta\cdot(x\times y)\] gives a ccm-magma on the set $$\{x\in A\mid \alpha\cdot 1+\beta\cdot x\, \text{is invertible}\}$$ for every choice of scalars $\alpha,\beta\in R$ with $\alpha=\alpha^2$.
\item Every loop with $a(bc)=c(ba)$ is a ccm-magma, see for instance \cite{Stein}.  
\item If $(A,\oplus)$ is a ccm-magma and $g\colon{A\to A}$ is a monomorphism, then the formula \[(x,y)\mapsto g(x\oplus y)\oplus a\] gives a ccm-magma on $A$, for every choice of $a\in A$.
\end{enumerate}
\end{example}

The last example is obtained directly from the proposition below by letting $f=k=1_A$ and $h=g$.

\begin{proposition}\label{prop2} Let $(A,\oplus)$ be a ccm-magma and $f,g\colon{A\to B}$ any two maps. If $f$ and $g$ are injective and there are maps $h,k\colon{A\to B}$ such that \[ f(g(x\oplus y)\oplus z)=(h(x)\oplus h(y))\oplus k(z), \quad x,y,z\in A,\] then the formula \[(x,y)\mapsto g(f(x)\oplus f(y))\oplus a\] gives a ccm-magma for every choice of $a\in A$.
\end{proposition}
\begin{proof}
The result is an immediate consequence of the axioms, combined with the hypotheses on the given maps. Commutativity is immediate. Cancellation follows from $f$ and $g$ being injective. To prove the medial law,  on the one hand we have \[f(g(fx\oplus fy)\oplus a)\oplus f(g(fz\oplus fw)\oplus a)\] which by our assumptions on $f$ and $g$ simplifies to \begin{equation}\label{eq1prop2}((h(fx)\oplus h(fy))\oplus k(a))\oplus ((h(fz)\oplus h(fw))\oplus k(a));\end{equation} while on the other hand we have \begin{equation*}
f(g(fx\oplus fz)\oplus a)\oplus f(g(fy\oplus fw)\oplus a),
\end{equation*} which by our assumptions on $f$ and $g$ simplifies to \begin{equation}\label{eq2prop2}
((h(fx)\oplus h(fz))\oplus k(a))\oplus ((h(fy)\oplus h(fw))\oplus k(a)).
\end{equation} Finally we observe that expressions $(\ref{eq1prop2})$ and $(\ref{eq2prop2})$ are equal by a simple manipulation of axiom (M3) on the original $\oplus$, which completes the proof.
\end{proof}

We end this section with two lemmas which will be used later on while proving that some internal relations, in the category of ccm-magmas, are automatically difunctional. This contrasts with the case of Mal'tsev categories in which every internal relation is difunctional, and every internal reflexive relation is always a congruence. 

\begin{lemma}\label{lemma1} Let $f,g\colon{A\times A\to B}$ be two morphisms in the category of ccm-magmas such that $f(a,a)=g(a,a)$ for every $a\in A$. Then:
\begin{enumerate}
\item[(i)] $f(a,b)=g(a,b),f(b,c)=g(b,c)\Rightarrow f(a,c)=g(a,c)$;
\item[(ii)] $f(a,b)=g(a,b)\Rightarrow f(b,a)=g(b,a)$.
\end{enumerate}
\end{lemma}
\begin{proof}
(i) We will show that if $f(a,b)=g(a,b)$ and $f(b,c)=g(b,c)$  then  we always have \begin{equation}\label{eq1lemma1}
f(a,c)\oplus f(b,b)=g(a,c)\oplus g(b,b)
\end{equation} and since $f(b,b)=g(b,b)$ we use (M2) and get $f(a,c)=g(a,c)$. To show (\ref{eq1lemma1}) we observe:
\begin{eqnarray*}
f(a,c)\oplus f(b,b)&=&f(a\oplus b,c\oplus b)\\
&=&f(a\oplus b,b\oplus c)\\
&=&f(a, b)\oplus f(b,c)\\
&=&g(a, b)\oplus g(b,c)\\
&=&g(a\oplus b,b\oplus c)\\
&=&g(a\oplus b,c\oplus b)\\
&=&g(a,c)\oplus g(b,b).
\end{eqnarray*}

(ii) Since $f(a,a)=g(a,a)$ for every $a\in A$, we always have:
\begin{eqnarray*}
f(b,a)\oplus f(a,b)&=&f(b\oplus a,a\oplus b)\\
&=&f(a\oplus b,a\oplus b)\\
&=&g(a\oplus b,a\oplus b)\\
&=&g(b\oplus a,a\oplus b)\\
&=&g(b,a)\oplus g(a,b)
\end{eqnarray*}
now, if $f(a,b)=g(a,b)$ then, using (M2) we obtain $f(b,a)=g(b,a)$ as desired.
\end{proof}

\begin{lemma}\label{lemma2} Let $f,g\colon{X\times Y\to B}$ be any two morphisms in the category of ccm-magmas. Then:
\[\left(\begin{matrix}f(x,y)=g(x,y)\\f(z,y)=g(z,y)\\f(z,w)=g(z,w)\end{matrix}\right) \Rightarrow f(x,w)=g(x,w).\]
\end{lemma}
\begin{proof}
Assuming $f(x,y)=g(x,y)$ and $f(z,w)=g(z,w)$ we have:
\begin{eqnarray*}
f(x,w)\oplus f(z,y)&=&f(x\oplus z,w\oplus y)\\
&=&f(x\oplus z,y\oplus w)\\
&=&f(x, y)\oplus f(z,w)\\
&=&g(x, y)\oplus g(z,w)\\
&=&g(x\oplus z,y\oplus w)\\
&=&g(x\oplus z,w\oplus y)\\
&=&g(x,w)\oplus g(z,y).
\end{eqnarray*}
Now, if $f(z,y)=g(z,y)$ then we conclude that $f(x,w)=g(x,w)$ as desired, using axiom (M2).
\end{proof}

\section{Mal'tsev and weakly Mal'tsev properties}\label{sec mal'tsev} 

We have already observed that the category of midpoint algebras, and hence the one of ccm-magmas, is not a Mal'tsev category. It is  a weakly Mal'tsev category as a result of ccm-magmas being a subcategory of commutative cancellative magmas, known to be weakly Mal'tsev. Nonetheless, we present an alternative proof using the medial law instead of commutativity.




\begin{proposition}The category of ccm-magmas is a weakly Mal'tsev category.
\end{proposition}
\begin{proof}
From \cite{NMF2012} we know that if there exits a ternary term $p(x,y,z)$ such that \[p(x,y,y)=p(y,y,x)\] and \[p(x,a,a)=p(y,a,a)\Rightarrow x=y\] then the category is weakly Mal'tsev. Indeed this is the case for every ccm-magma $(A,\oplus)$ if we define \[p(x,y,z)=(y\oplus x)\oplus(z\oplus y).\qedhere\]
\end{proof}

In a weakly Mal'tsev category any diagram  of the shape
\begin{equation}\label{couniv}
\vcenter{\xymatrix@!0@=4em{A \ar@<.5ex>[r]^-{f} \ar[rd]_-{u} & B
\ar@<.5ex>[l]^-{r}
\ar@<-.5ex>[r]_-{s}
\ar[d]^-{v} & C \ar@<-.5ex>[l]_-{g} \ar[ld]^-{w}\\
& D}}
\end{equation}
with $fr=1_{B}=gs$ and $u r=v=w s$, induces a bigger diagram
\begin{equation}\label{kite}
\vcenter{\xymatrix@!0@=3em{ & C \ar@<.5ex>[ld]^-{e_2} \ar@<-.5ex>[rd]_-{g}
\ar@/^/[rrrd]^-{w} \\
A\times_{B}C \ar@<.5ex>[ru]^-{\pi_2} 
 \ar@<-.5ex>[rd]_-{\pi_1} && B \ar@<.5ex>[ld]^-{r} \ar@<-.5ex>[lu]_-{s}
 \ar[rr]|-{v} && D\\
& A \ar@<.5ex>[ru]^-{f} \ar@<-.5ex>[lu]_-{e_1} \ar@/_/[urrr]_-{u}}}\end{equation}
in which there is at most one morphism, say $\theta\colon{A\times_B C\to D}$, from the pullback $(A\times_B C,\pi_1,\pi_2)$ of the split epimorphism $f$ along the split epimorphism $g$ to the object $D$ such that $\theta e_1=u$ and $\theta e_2=w$. Where $e_1$ and $e_2$ are the induced morphisms of the form\[e_1(a)=(a,sf(a)),\quad e_2(c)=(rg(c),c).\]

The following result states under which conditions a diagram such as (\ref{couniv}) induces a morphism such as $\theta$.

\begin{proposition} Given a diagram such as $(\ref{kite})$ in the category of ccm-magmas, there is a (unique) morphism $\theta\colon{A\times_B C\to D}$, such that $\theta e_1=u$ and $\theta e_2=w$, if and only if,  the equation \begin{equation}\label{equation xvuw}x\oplus v(b)=u(a)\oplus w(c)\end{equation} has a solution $x\in D$  for every $a\in A$, $b\in B$ and $c\in C$ with $f(a)=b=g(c)$. When that is the case, $\theta(a,c)=x$. 
\end{proposition}
\begin{proof}
Suppose there is $\theta\colon{A\times_B C\to D}$ such that $\theta(a,s(b))=u(a)$ and $\theta(r(b),c)=w(c)$, where $b=f(a)=g(c)$, then $x=\theta(a,c)$ is always a solution to the equation (\ref{equation xvuw}). Indeed, using axiom (M1), we observe that \begin{eqnarray*}
\theta(a,c)\oplus v(b)&=&\theta(a,c)\oplus \theta(r(b),s(b))\\
&=& \theta((a,c)\oplus(r(b),s(b)))\\
&=& \theta(a\oplus r(b), c\oplus s(b))\\
&=& \theta(a\oplus r(b), s(b)\oplus c)\\
&=& \theta(a,s(b))\oplus \theta(r(b),c)\\
&=& u(a)\oplus w(c).
\end{eqnarray*}
Conversely, if the equation (\ref{equation xvuw}) has a solution (which is unique by axiom (M2)) for every $(a,c)\in A\times_B C$, then we may define a map $\theta\colon{A\times_B C\to D}$ which assigns the unique solution of (\ref{equation xvuw}) to every pair $(a,c)\in A\times_B C$. It remains to show that this map is a homomorphism. By axiom (M2), it suffices to prove that \[(\theta(a,c)\oplus\theta(a',c'))\oplus v(b\oplus b')=\theta(a\oplus a',c\oplus c')\oplus v(b\oplus b')\] for every $a,a'\in A$  and $c,c'\in C$ where \[f(a)=g(c)=b\in B \text{ and } f(a')=g(c')=b\in B.\] Now, because $u$, $v$ and $w$ are homomorphisms and,  using axiom (M3), we have \begin{eqnarray*}
(\theta(a,c)\oplus\theta(a',c'))\oplus v(b\oplus b')&=&(\theta(a,c)\oplus\theta(a',c'))\oplus (v(b)\oplus v(b'))\\
&=&(\theta(a,c)\oplus v(b))\oplus (\theta(a',c')\oplus v(b'))\\
&=&(u(a)\oplus w(c))\oplus (u(a')\oplus w(c'))\\
&=&(u(a)\oplus u(a'))\oplus (w(c)\oplus w(c'))\\
&=& u(a\oplus a')\oplus w(c\oplus c')\\
&=&\theta(a\oplus a',c\oplus c')\oplus v(b\oplus b')
\end{eqnarray*} as desired, which completes the proof.
\end{proof}

In particular, any internal reflexive graph admits, at most, one structure of internal category. This is easily seen from the above result by choosing $D=A=C$, $r=s=v$, and $u$ and $w$ to be the identity morphisms. Even more particularly, by choosing $B$ to be a singleton, and if $(A,\oplus)$ is a ccm-magma then, for every idempotent $e\in A$, there is, at most, one internal monoid structure on $A$ which is compatible with the binary operation, that is, there exists at most one monoid $(A,*_e,e)$ such that \begin{equation}\label{condition * compatible with oplus}
(x*_ey)\oplus(z*_e w)=(x\oplus z)*_e(y\oplus w).
\end{equation}

Note that $e\in A$  must be an idempotent, so that the inclusion $\{e\} \to A$ may be a homomorphism.

\begin{corollary}\label{corollary1} Let $(A,\oplus)$ be a ccm-magma and $e\in A$ an idempotent element in $A$. There is a (unique) internal monoid structure $(A,*_e,e)$ in $A$, if and only if the equation \[\theta\oplus e=x\oplus y\] has a solution $\theta=\theta(x,y)\in A$ for every $x,y\in A$. In that case $x*_e y$ is given by $\theta(x,y)$.
\end{corollary}
\begin{proof}
It follows from the previous proposition with $A=D=C$, $B=1$ the terminal algebra, $f$ and $g$ uniquely determined while $r=s=v$ send the unique element in $1$ to the chosen element $e$ in $A$ (which is a homomorphism as soon as $e\oplus e=e$), and $u=w$ is the identity morphism. The fact that the operation $*_e$ is associative and has a unit $e$ follows from general arguments used in \cite{NMF2008} but may also be demonstrated directly.
\end{proof}

A particular case is when every element $a\in A$ can be decomposed as ${a=x_1\oplus x_2}$, studied in \cite{Jezek93} as division groupoids. In this case the property of having an internal monoid structure with unit element $e$ is equivalent to asking for a  solution to the equation $x\oplus e=a$. We will further study these properties in the next section.

\section{Internal monoids and internal groups}\label{sec monoids}

It is well known \cite{CKP,Carboni-Lambek-Pedicchio,Carboni-Pedicchio-Pirovano} that, in Mal'tsev categories, every internal monoid is necessarily and internal group. This property does not apply to ccm-magmas. In this section, we will study some sufficient conditions for the existence of an internal monoid or group structure within a ccm-magma with a chosen unit element, which is necessarily an idempotent. For that we introduce the following notions, which have already been considered in the literature for different purposes, see for example \cite{Aczel-book,Jezek}:

\begin{definition}\label{def_homogeneous} Let $(A,\oplus)$ be a ccm-magma and consider any element $e\in A$. We will say that:
\begin{enumerate}
\item[(i)] $A$ is $e$-expansive if for every $a\in A$ there exists $2_e(a)\in A$ such that $2_e(a)\oplus e=a$;
\item[(ii)] $A$ is $e$-symmetric if for every $a\in A$ there exists $-_e(a)\in A$ such that $-_e(a)\oplus a=e$;
\item[(iii)] $A$ is homogeneous if it is $e$-expansive (or $e$-symmetric) for every $e\in A$.
\end{enumerate}
\end{definition}

In fact, a homogeneous ccm-magma is the same as a commutative medial quasi\-group (see for instance \cite{Leibak}).

A sufficient condition for a ccm-magma to admit an internal monoid structure with an idempotent element $e$ as its unit is to be $e$-expansive. When that is the case, then the internal monoid is a group if and only if the algebra is $e$-symmetric. Moreover, if every element is an idempotent, that is, if we have a midpoint algebra, then it is $e$-expansive if and only if there exists a monoid structure over $e$. Also every internal monoid (or group) is commutative and admits cancellation.

\begin{proposition}\label{prop expansive implies has a monoid} Let $(A,\oplus)$ be a ccm-magma and consider any idempotent element $e\in A$. If $A$ is $e$-expansive then $(A,*_e,e)$ is a monoid with $$x*_ey=2_e(x\oplus y).$$ Moreover it is a group if and only if $A$ is $e$-symmetric.
\end{proposition}
\begin{proof} If $A$ is $e$-expansive then in particular  $2_e(x\oplus y)$ is a solution to the equation \[\theta\oplus e=x\oplus y,\] for every $x,y\in A$. From Corollary \ref{corollary1} we may conclude that $(A,*_e,e)$ is an internal monoid. Now, if moreover $A$ is $e$-symmetric then $-_e(a)$ is the inverse of $a$, for every $a\in A$. Indeed we have that \[-_e(a)*_ea=e\]is equivalent (by (M2)) to\[(-_e(a)*_e a)\oplus e=e\oplus e\] which holds because \begin{eqnarray*}
(-_e(a)*_e a)\oplus e  =  2_e(-_e(a)\oplus a)\oplus e=-_e(a)\oplus a=e=e\oplus e.
\end{eqnarray*}
  Conversely, if $(A,*_e,e)$ is a group, then, for every $a\in A$, its symmetric element, say $a'\in A$, is a solution to the equation $x\oplus a=e$. Indeed, since $a'$ is such that $a'*_e a=e$, or equivalently $2_e(a'\oplus a)=e$, we use axiom (M2) and the assumption that $e$ is an idempotent to conclude that \[e=e\oplus e= 2_e(a'\oplus a)\oplus e=a'\oplus a.\qedhere\]
\end{proof}

In the case when the operation $\oplus$ has the geometrical meaning of midpoint, the formula $a*_e b=2_e(a\oplus b)$ is intuitively illustrated via the following diagram.
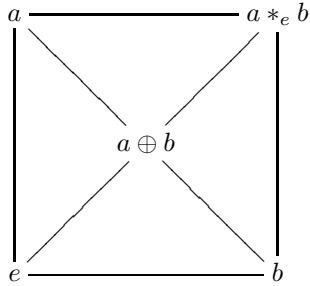
\[\xymatrix@!@=2em{a\ar@{-}[rr]\ar@{-}[dd]\ar@{-}[rd]&&a*_eb\ar@{-}[ld]\ar@{-}[dd]\\&a\oplus b \ar@{-}[ld]\ar@{-}[rd]\\e\ar@{-}[rr]&&b}\]
As a more concrete example, let $A$ be any real vector space. Then, by defining \[a\oplus b=\frac{1}{2}(a+b)\] we obtain a ccm-magma which is $0$-symmetric and $0$-expansive with the usual interpretation of $-a$ and $2a$, as illustrated for the particular case of the real line.
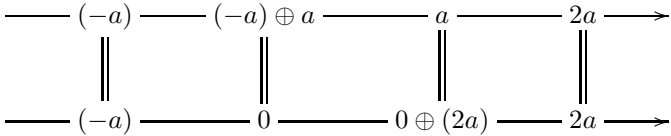
\[\xymatrix{ \ar@{-}[r] & (-a)\ar@{=}[d]\ar@{-}[r] & (-a)\oplus a\ar@{=}[d] \ar@{-}[r] & a \ar@{=}[d] \ar@{-}[r] & 2a\ar@{=}[d] \ar@{->}[r] &\\\ar@{-}[r] & (-a)\ar@{-}[r] & 0 \ar@{-}[r] & 0\oplus (2a) \ar@{-}[r] & 2a \ar@{->}[r] &}\]
More generally, for every $e\in A$, this structure of ccm-magma is $e$-symmetric and $e$-expansive, with $2_e(a)=2a-e$ and $-_e(a)=2e-a$. This fact is related to the affine transformation $x\mapsto x+e$.

We also notice that if a ccm-magma is $e$-expansive and $e$-symmetric, for some  element $e$, then it is so for all elements, in other words it is homogeneous. This result will be used in the proof of Corollary \ref{corollary3}.

\begin{proposition}\label{prop_homogeneous} Let $(A,\oplus)$ be a ccm-magma and $e\in A$ an element in it. If it is $e$-expansive and $e$-symmetric then it is homogeneous.
\end{proposition}
\begin{proof}
We have to prove that for every $u,v\in A$, there exists a solution $x$ to the equation $x\oplus u=v$. Indeed,
\[\begin{array}{rcl}
x\oplus u=v&\Longleftrightarrow&
(x\oplus u)\oplus(e\oplus -_e(u))=v\oplus(e\oplus -_e(u))\\
&\Longleftrightarrow&
(x\oplus e)\oplus(u\oplus -_e(u))=v\oplus(e\oplus -_e(u))\\
&\Longleftrightarrow&
(x\oplus e)\oplus e=v\oplus(e\oplus -_e(u))\\
&\Longleftrightarrow&
x\oplus e=2_e(v\oplus(e\oplus -_e(u)))\\
&\Longleftrightarrow&
x=2_e(2_e(v\oplus(e\oplus -_e(u)))),
\end{array}\] which gives  the desired solution to the equation. 
\end{proof}

As already referred, every homogeneous ccm-magma is a commutative medial quasigroup. This means that for homogeneous ccm-magmas, Proposition \ref{prop expansive implies has a monoid} is a special case (commutative) of the well known Toyoda Theorem \cite{Toyoda}. This theorem has been generalized for medial magmas with cancellation (see for instance \cite{Jezek}), but  the result is no longer comparable with the one of an internal monoid structure.

Restricting the study to homogeneous ccm-magmas has some advantages but it forces the category to be  Mal'tsev (and hence there is no longer the distinction between internal monoid and internal group). Indeed, adapting the well-known formulas describing the category of quasigroups as a Mal'tsev category, say from \cite{Edward},  we can conclude that if a ccm-magma is expansive for every element then \[p(x,y,z)=2_{2_y(y)}(x)\oplus 2_y(z)\] is a Mal'tsev term.

We continue with another aspect of ccm-magmas  which will be used in Proposition \ref{prop_iso_monoids}. If a given ccm-magma is expansive with respect to some idempotent element $e$ then the respective map $2_e$ is a homomorphism. In general, we have:

\begin{proposition}\label{prop_exp_symm_morphism}
Let $(A,\oplus)$ be a ccm-magma. If it is $u$-expansive and $v$-expansive for some $u,v\in A$, then it is also $(u\oplus v)$-expansive, and moreover, \[2_u(a)\oplus 2_v(b)=2_{u\oplus v}(a\oplus b).\]
\end{proposition}
\begin{proof}
Suppose there exists $2_{u\oplus v}(a\oplus b)$, then we necessarily have $$2_{u}(a)\oplus 2_{v}(b)=2_{u\oplus v}(a\oplus b)$$ because, by \emph{adding} $u\oplus v$ in each term, we obtain $a\oplus b$. It remains to prove that $2_{u\oplus v}$ exists. In other words we have to prove that for every $a\in A$, there is $x\in A$ such that $x\oplus (u\oplus v)=a$. We claim that \[x=2_u(2_u(a))\oplus 2_v(u)\]is the needed solution. Indeed,
\begin{eqnarray*}
(2_u(2_u(a))\oplus 2_v(u))\oplus(u\oplus v)&=&(2_u(2_u(a))\oplus u)\oplus(2_v(u)\oplus v)\\
&=&2_u(a)\oplus u\\
&=& a.
\end{eqnarray*}
So, we get $2_{u\oplus v}(a)=2_u(2_u(a))\oplus 2_v(u)=2_v(2_v(a))\oplus 2_u(v)$.
\end{proof}

The next result explains the connection between two induced monoid structures for two different idempotents $u$ and $v$.

\begin{proposition}\label{prop_iso_monoids} Let $(A,\oplus)$ be a ccm-magma which is $u$-expansive and $v$-expansive for some idempotent elements $u,v\in A$. The two monoid structures on $A$, induced by $u$ and $v$, are isomorphic. Moreover the two structures are related as follows:  \[a*_u b= (a*_v b)*_u v.\]
\end{proposition}
\begin{proof}
Let $a*_u b=2_u(a\oplus b)$ and $a*_v b=2_v(a\oplus b)$ be the two monoid operations induced, respectively by $u$ and $v$, assuming $u$ and $v$ to be idempotent and $A$ to be $u$-expansive and $v$-expansive.
 The map $f\colon{(A,*_u,u)\to (A,*_v,v)}$,  such that $f(a)=2_u(a\oplus v)$, $a\in A$, is a monoid homomorphism. Clearly, the units are preserved, since \[f(u)=2_{u}(u\oplus v)=2_{u}(v\oplus u)=2_u(v)\oplus 2_u(u)=2_u(v)\oplus u=v.\] From
\begin{eqnarray*}
f(a*_u b)\oplus v&=&f(2_u(a\oplus b))\oplus v\\
&=&2_u(2_u(a\oplus b)\oplus v)\oplus  v\\
&=&2_u(2_u(a\oplus b)\oplus v)\oplus (u\oplus 2_u(v))\\
&=&(2_u(2_u(a\oplus b))\oplus 2_u(v))\oplus (u\oplus 2_u(v))\\
&=&(2_u(2_u(a\oplus b))\oplus u)\oplus (2_u(v)\oplus 2_u(v))\\
&=&2_u(a\oplus b)\oplus (2_u(v\oplus v)
\\&=&2_u(a\oplus b)\oplus (2_u(v)\\
&=&2_u((a\oplus b)\oplus v)
\end{eqnarray*}
and 
\begin{eqnarray*}
(f(a)*_v f(b))\oplus v&=&2_v(f(a)\oplus f(b))\oplus v\\
&=&f(a)\oplus f(b)\\
&=&2_u(a\oplus v)\oplus 2_u(b\oplus v)\\
&=&2_u((a\oplus v)\oplus (b\oplus v))\\
&=&2_u((a\oplus b)\oplus (v\oplus v))\\
&=&2_u((a\oplus b)\oplus v)
\end{eqnarray*}
we may conclude that $f(a*_u b)=f(a)*_v f(b)$. The inverse homomorphism of $f$ is $g\colon{(A,*_v,v)\to (A,*_u,u)}$ with $g(a)=2_v(a\oplus u)$. Indeed,
\[gf(a)=2_v(2_u(a\oplus v)\oplus u)=2_v(a\oplus v)=2_v(a)\oplus 2_v(v)=2_v(a)\oplus v=a\]
and similarly we prove $fg(a)=a$.
Finally, we prove $a*_u b= (a*_v b)*_u v$ by observing that \[(a*_u b)\oplus u=a\oplus b= ((a*_v b)*_u v)\oplus u.\qedhere\]
\end{proof}

In some cases, a ccm-magma $(A,\oplus)$ does not only admit an internal monoid structure over some idempotent element $e$, but also the structure itself is a monoid with unit element $e$, that is $a\oplus e=a$. This property is summarized in the following proposition.

\begin{proposition}\label{prop ccm-magma associative}
Let $(A,\oplus)$ be a ccm-magma and $e\in A$ an idempotent. The following conditions are equivalent:
\begin{enumerate}
\item[(i)] the operation $\oplus$ is associative;
\item[(ii)] the element $e\in A$ is a unit element for $\oplus$;
\item[(iii)] the ccm-magma is $e$-expansive with $2_e(a)=a$;
\item[(iv)] the structure $(A,\oplus, e)$ is an internal monoid.
\end{enumerate}
\end{proposition}
\begin{proof}
If the operation $\oplus$ is  associative and $e\in A$ is idempotent then \[(a\oplus e)\oplus e=a\oplus (e\oplus e)=a\oplus e\] and hence $a\oplus e=a$.
If $e\in A$, an idempotent, is also a unit element for $\oplus$ then $2_e(a)=a$ by definition of $2_e$.
We already know that if the ccm-magma is $e$-expansive then $(A,*_e,e)$ is an internal monoid structure with $x*_e y =2_e(x\oplus y)$. When $2_e$ is the identity map, the operation $*_e$ is simply the original $\oplus$.
This proves $\text{(i)}\Rightarrow\text{(ii)}\Rightarrow \text{(iii)}\Rightarrow\text{(iv)}$, whereas $\text{(iv)}\Rightarrow \text{(i)}$ is obvious. 
\end{proof}

Note also that if a ccm-magma $(A,\oplus)$ is associative, then it has, at most, one idempotent. Indeed, if $e_1$ and $e_2$ are two idempotents, and $\oplus$ is associative, then \[e_1\oplus (e_1\oplus e_2)=e_1\oplus e_2=e_2\oplus(e_2\oplus e_1)\]from which it follows that $e_1=e_2$.

In the case of midpoint algebras, that is, when every element is an idempotent, there is no distinction between having a monoid structure for some element $e\in A$ and being $e$-expansive.

\begin{proposition}\label{prop midpoint and expansive}
Let $(A,\oplus)$ be a midpoint algebra with $e\in A$. The following conditions are equivalent:
\begin{enumerate}
\item[(i)] there is an internal monoid structure over $e$;
\item[(ii)] for every $a,b\in A$, there is $x\in A$ such that $x\oplus e=a\oplus b$;
\item[(iii)] the midpoint algebra is $e$-expansive.
\end{enumerate}
\end{proposition}
\begin{proof} The two conditions (i) and (ii) are already equivalent  in the context of ccm-magmas (Corollary \ref{corollary1}). Also, in the more general case of ccm-magmas, (iii) implies (i), as it is proven in Proposition \ref{prop expansive implies has a monoid}. We are left to prove that (i) implies (iii). Assuming an internal monoid structure $(A,*_e,e)$ and the fact that every element  $a\in A$ is idempotent ($A$ is a midpoint algebra), we have \[2_e(a)=2_e(a\oplus a)=a*_e a.\qedhere\]
\end{proof}

If, in a midpoint algebra $(A,\oplus)$, there is an internal monoid structure $(A,*_e,e)$, then $*_e$ is distributive over $\oplus$, as it immediately follows from $(\ref{condition * compatible with oplus})$. This is not true in general for ccm-magmas.

\begin{proposition}
Let $(A,\oplus)$ be a ccm-magma with an internal monoid structure $(A,*_e,e)$. The following two conditions are equivalent:
\begin{enumerate}
\item[(i)] the ccm-magma is a midpoint algebra (every element $a\in A$ is an idempotent);
\item[(ii)] for every $x,y,z\in A$, \begin{equation}\label{distribituvity}
x*_e(y\oplus z)=(x*_e y)\oplus (x*_e z).
\end{equation}
\end{enumerate}
\end{proposition}
\begin{proof} If $(A,\oplus)$ is a midpoint algebra then   $(\ref{distribituvity})$ follows from $(\ref{condition * compatible with oplus})$.
Conversely, if we have an internal monoid structure $(A,*_e,e)$ together with $(\ref{distribituvity})$, then every element $a\in A$ is an idempotent  \[a\oplus a=(a*_e e)\oplus(a*_e e)=a*_e(e\oplus e)=a*_e e=a.\] Note that  $e$ is  an idempotent because it is the unit element of an internal monoid.
\end{proof}


\subsection*{Some examples illustrating the properties discussed in the previous results}

This section concludes with a short list of examples and counter-examples showing several particular cases and properties, which a given ccm-magma may or may not have for some specific choice of an idempotent element $e$ in it. First we observe that if the four basic properties in the left column on the table below are considered, then there are only six possible combinations between them, namely the ones expressed in the other columns and denoted by I to VI:
\begin{center}
\begin{tabular}{|c|c|c|c|c|c|c|}
\hline
 & I & II & III & IV & V & VI\\
\hline 
$e$-expansive & yes & yes & no  & no & no & no \\
\hline 
$e$-symmetric & yes & no & yes  & no & no & yes \\
\hline 
internal monoid over $e$ & yes & yes & no  & no & yes & yes \\
\hline 
internal group over $e$ & yes & no & no  & no & no & yes \\
\hline 
\end{tabular} 
\end{center}
The properties I to VI defined in the table above have an obvious interpretation. For example, a ccm-magma has property III if and only if it is $e$-symmetric but not $e$-expansive for a specific choice of $e$, it does not have an internal group structure over $e$ or even an internal monoid; while the ccm-magmas with property V have an internal monoid structure over some idempotent element $e$, they are not $e$-expansive nor $e$-symmetric, and consequently do not posses an internal group structure.
Now, combining the previous properties with the existence of one or more,  none, or even all idempotents, and also with associativity, we observe the following list of ccm-magmas $(A,\oplus)$.
\begin{enumerate}
\item  Ccm-magmas with no idempotent elements (in this case there is no interaction with properties I to VI from above):
\begin{enumerate}
\item Non-associative:
\begin{enumerate}
\item $a\oplus b=\frac{a+b}{2}+1$, $A=\mathbb{R}$
\item $a\oplus b=\frac{3ab}{a+b}$, $A=\mathbb{R}^{+}$
\item $a\oplus b=2(a+b)$, $A=\mathbb{R}^{+}$
\end{enumerate}
\item Associative:
\begin{enumerate}
\item $a\oplus b=a+b+1$, $A=[0,+\infty]$
\item $a\oplus b=\frac{ab}{a+b}$, $A=\mathbb{R}^{+}$ or $A=]0,1]$
\item $a\oplus b=a+b+ab$, $A=\mathbb{R}^{+}$
\item $a\oplus b=a+b$, $A=\mathbb{R}^{+}$
\item $a\oplus b=\frac{a+b}{1+ab}$, $A=]0,1[$
\item $a\oplus b=\log(\exp(a)+\exp(b))$, $A=\mathbb{R}$
\end{enumerate}
\end{enumerate}
\item Ccm-magmas with every element  an idempotent, that is, midpoint algebras (due to Proposition \ref{prop midpoint and expansive} the properties V and VI do not apply):
\begin{center}
\begin{tabular}{|c|c|c|c|c|}
\hline
Property & $a\oplus b$ & $A$ & $e\in A$\\
\hline 
I & $\frac{a+b}{2}$ &  $\mathbb{R}$ & every\\  
\hline 
I & $\sqrt[3]{\frac{a^3+b^3}{2}}$ &  $\mathbb{R}$ & every\\  
\hline 
II & $\frac{a+b}{2}$ &  $[0,+\infty[$ &    $e=0$\\  
\hline 
II & $\frac{2ab}{a+b}$ &  $]0,1]$ &    $e=1$\\  
\hline 
III & $\frac{a+b}{2}$ &  $[0,1]$ & $e=\frac{1}{2}$\\  
\hline 
III & $\frac{2ab}{a+b}$ &  $]1,+\infty[$ &    $e=2$\\  
\hline 
IV & $\frac{a+b}{2}$ &  $\mathbb{R}^{+}$ & N/A\\  
\hline 
IV & $\frac{2ab}{a+b}$ &  $\mathbb{R}^{+}$ & N/A\\  
\hline 
\end{tabular} 
\end{center}
\item Ccm-magmas with at least one idempotent element which are not mid\-point-algebras (in this case again we distinguish between associative and non-associative and give examples of each one of the properties I to VI from above):
\begin{enumerate}
\item non-associative:
\begin{center}
\begin{tabular}{|c|c|c|c|c|}
\hline
Property & $a\oplus b$ & $A$ & $e\in A$\\
\hline 
I & $2(a+b)$ &  $\mathbb{R}$ & $e=0$\\  
\hline 
I & $2\sqrt[3]{a^3+b^3}$ &  $\mathbb{R}$ & $e=0$\\  
\hline 
II & $2(a+b)$ &  $[0,+\infty[$ &    $e=0$\\  
\hline 
III & $\frac{a+b}{3}$ &  $[-1,1]$ &    $e=0$\\  
\hline 
IV & $\frac{a+b}{3}$ &  $[0,1]$ & $e=0$\\  
\hline 
V & $2(a+b)$ &  $\mathbb{N}_{0}$ &    $e=0$\\  
\hline 
VI & $2(a+b)$ &  $\mathbb{Z}$ & $e=0$\\  
\hline 
\end{tabular} 
\end{center}
\item associative (in this case we only distinguish properties I and II from above, as it follows from Proposition \ref{prop ccm-magma associative}).
\begin{center}
\begin{tabular}{|c|c|c|c|c|}
\hline
Property & $a\oplus b$ & $A$ & $e\in A$\\
\hline 
I & $a+b$ &  $\mathbb{R}$ & $e=0$\\  
\hline 
II & $a+b$ &  $[0,+\infty[$ &    $e=0$\\  
\hline 
II & $a+b-ab$ &  $[0,1[$ &    $e=0$\\  
\hline 
\end{tabular} 
\end{center}
\end{enumerate}

\item Finite ccm-magmas. Every finite ccm-magma is homogeneous. Some examples are as follows. The matrix \[A_2\left(\begin{array}{ccc} 1 & 3 & 2\\ 3 & 2 & 1\\ 2 & 1 & 3 \end{array}\right)\] shows an example of the \emph{multiplication table} for a non-associative ccm-magma with three idempotents. The matrix\[A_3
\left(\begin{array}{ccc} 2 & 1 & 3\\ 1 & 3 & 2\\ 3 & 2 & 1 \end{array}\right)\] shows an example of a non-associative ccm-magma with no idempotents. However there are no ccm-magmas with a finite and even number of idempotents. This is due to the fact that if $(A,\oplus)$ is a ccm-magma then the subset \mbox{$\{a\in A\mid a\oplus a=a\}$} is a subalgebra of ccm-magmas and there are no commutative idempotent quasigroups (homogeneous ccm-magmas) of even order \cite{Stein}.
\end{enumerate}

\section{Internal Relations}\label{sec relations}

As we have observed in the introduction, the category of ccm-magmas is not Mal'tsev, and so there is no hope of having every internal reflexive relation automatically as a congruence, nor having any internal relation as a difunctional one. Nevertheless, as it is shown in \cite{ZJ-NMF} a category is weakly Mal'tsev if and only if every strong relation is difunctional. In particular, if $f,g\colon{X\times Y\to B}$ are any two morphisms of ccm-magmas, then the relation $R\subseteq X\times Y$, defined by
\begin{equation}\label{eq_strongrel}
xRy\Leftrightarrow f(x,y)=g(x,y)\end{equation}
 is a strong relation.

Hence the following results:

\begin{proposition}\label{prop_intrel1}  Let $f,g\colon{A\times A\to B}$ be two morphisms in the category of ccm-magmas such that $f(a,a)=g(a,a)$ for every $a\in A$. The internal relation defined by \[xRy\Leftrightarrow f(x,y)=g(x,y)\] is a congruence.
\end{proposition}
\begin{proof}
It is an immediate consequence of Lemma \ref{lemma1}.
\end{proof}

\begin{proposition}\label{prop_intrel2}  Let $f,g\colon{X\times Y\to B}$ be two morphisms in the category of ccm-magmas. The internal relation defined by \[xRy\Leftrightarrow f(x,y)=g(x,y)\] is difunctional.
\end{proposition}
\begin{proof}
It is an immediate consequence of Lemma \ref{lemma2}.
\end{proof}

We now present some results involving another special type of internal relations, namely the ones constructed from a subalgebra of a given ccm-magma.

\begin{proposition}\label{prop R rel items} Let $(A,\oplus)$ be a ccm-magma with a subalgebra $X\subseteq A$ and an idempotent element $e\in X$. The relation
\begin{equation}\label{relationsubalg}aRb\Leftrightarrow \exists x\in X, a\oplus e=x\oplus b\end{equation}
 \begin{enumerate}
\item[(i)]  is an internal relation;
\item[(ii)]  is reflexive;
\item[(iii)] is transitive whenever $X$ has a (unique) internal monoid structure with $e\in X$ as its unit;
\item[(iv)] is symmetric if and only if $X$ is $e$-symmetric.
\end{enumerate}
\end{proposition}
\begin{proof}
(i) The relation is an internal relation if and only if $R\subseteq A\times A$ is a subalgebra of the product, or equivalently, if and only if\[aRb,a'Rb'\Rightarrow (a\oplus a')R(b\oplus b').\] It is the case because if there exist $x,x'\in X$ such that $a\oplus e=x\oplus b$ and $a'\oplus e=x'\oplus b'$ then  we have \begin{eqnarray*}
(a\oplus a')\oplus e &=&  (a\oplus a')\oplus (e\oplus e)\\
&=& (a\oplus e)\oplus (a'\oplus e)\\
&=& (x\oplus b)\oplus (x'\oplus b')\\
&=& (x\oplus x')\oplus (b\oplus b')\\
\end{eqnarray*}
showing  that  there is $(x\oplus x')\in X$ such that\[(a\oplus a')\oplus e=(x\oplus x')\oplus (b\oplus b'),\] and so $(a\oplus a')R(b\oplus b')$.

(ii) The reflexivity of $R$ follows from the observation that \mbox{$a\oplus e=e\oplus a$} and $e\in X$.

(iii) Suppose there is an internal monoid structure in $X$ with $e$ as its unit element. Corollary \ref{corollary1} tells us that $X$ has a monoid structure with $e$ as unit if and only if for every $x,y\in X$ there is $(x*_e y)\in X$ such that $(x*_e y)\oplus e=x\oplus y$. In this case, we prove transitivity by showing that if $aRb$ and $bRc$, that is \[a\oplus e=x\oplus b\quad,\quad b\oplus e=y\oplus c,\] then $aRc$, because \[a\oplus e=(x*_e y)\oplus c.\]
It is straightforward to prove the above equality, by composing with $(e\oplus e)$  and then using cancellation:
\begin{eqnarray*}
(a\oplus e)\oplus (e\oplus e)&=&(x\oplus b)\oplus (e\oplus e)\\
&=&(x\oplus e)\oplus (b\oplus e)\\
&=&(x\oplus e)\oplus (y\oplus c)\\
&=&(x\oplus y)\oplus (e\oplus c)\\
&=&((x*_e y)\oplus e)\oplus (c\oplus e)\\
&=&((x*_e y)\oplus c)\oplus (e\oplus e).
\end{eqnarray*} 

(iv) Let us show that $R$ is symmetric if and only if $X$ is $e$-symmetric.
By definition of $R$, for every $x\in X$, we have $xRe$. Hence, if the relation is symmetric we also have $eRx$ from which we conclude the existence of $y\in X$ such that $e=e\oplus e=y\oplus x$. This shows that $X$ is $e$-symmetric. Conversely, if $X$ is $e$-symmetric then for every $x\in X$ there is $y\in X$ such that $x\oplus y=e$ and consequently the relation $R$ is symmetric. Indeed, given $aRb$, that is $a\oplus e=x\oplus b$ for some $x\in X$,  we have, with $y=-_e(x)$: \[(y\oplus a)\oplus e=(y\oplus e)\oplus(a\oplus e)=(y\oplus x)\oplus(e\oplus b)=e\oplus(e\oplus b)\]from which we conclude \[y\oplus a=b\oplus e\] and so $bRa$.
\end{proof}

The condition on the existence of a monoid structure in item (iii) above is sufficient for the relation to be transitive but it is not necessary. Indeed we only have the operation $(x*_e y)$ well defined for certain  pairs $(x,y)\in X\times X$, namely the ones for which given any $c\in A$ there are $a,b\in A$ such that $a=(y*_e c)$ and $b=x*_e (y*_e c)$.

This suggests the following necessary and sufficient condition for the transitivity of this type of internal relations. 

\begin{proposition} Let $(A,\oplus)$ be a ccm-magma with a subalgebra $X\subseteq A$ and an idempotent element $e\in X$. The relation\[aRb\Leftrightarrow \exists x\in X, a\oplus e=x\oplus b\]
 
is transitive if and only if:\begin{quote}
for all $x,y\in X$ and $c\in A$, if there exist $a,b\in A$ such that\begin{equation}\label{eq used for transitivity}
a\oplus e=x\oplus b\quad\text{and}\quad b\oplus e=y\oplus c
\end{equation}then there is $z\in X$ such that $z\oplus e=x\oplus y$.
\end{quote}
\end{proposition}
\begin{proof}
Assume $R$ is transitive. If we have solutions $a$ and $b$ for the equations (\ref{eq used for transitivity}) then we also have $aRb$ and $bRc$ which, by transitivity, gives us the desired $z\in X$ such that $a\oplus e=z\oplus c$. It is now a simple calculation to check that $z\oplus e=x\oplus y$. Indeed we have
\begin{eqnarray*}
(z\oplus e)\oplus (c\oplus e)&=&(z\oplus c)\oplus (e\oplus e)\\
&=&(a\oplus e)\oplus (e\oplus e)\\
&=&(x\oplus b)\oplus (e\oplus e)\\
&=&(x\oplus e)\oplus (b\oplus e)\\
&=&(x\oplus e)\oplus (y\oplus c)\\
&=&(x\oplus y)\oplus (e\oplus c)
\end{eqnarray*} 
and the result follows from cancellation.
Conversely, if $aRb$ and $bRc$ then we also have $x,y\in X$ as in  $(\ref{eq used for transitivity})$ and hence  there is an element $z\in X$ such that $z\oplus e=x\oplus y$ from which we conclude that  $a\oplus e=z\oplus c$. Indeed
\begin{eqnarray*}
(a\oplus e)\oplus (e\oplus e)&=&(x\oplus b)\oplus (e\oplus e)\\
&=&(x\oplus e)\oplus (b\oplus e)\\
&=&(x\oplus e)\oplus (y\oplus c)\\
&=&(x\oplus y)\oplus (e\oplus c)\\
&=&(z\oplus e)\oplus (c\oplus e)\\
&=&(z\oplus c)\oplus (e\oplus e),
\end{eqnarray*} 
this shows that $aRc$ and concludes the proof.
\end{proof}

When $A$ is $e$-expansive, item (iii) in Proposition \ref{prop R rel items}  can now be reformulated so to relate the transitivity of $R$ with the property of $X$ being $e$-expansive.

\begin{corollary}\label{corollary2} Let $(A,\oplus)$ be a ccm-magma which is $e$-expansive for some idempotent element $e\in A$ and let $X$ be a subalgebra with $e\in X$. The relation \[aRb\Leftrightarrow \exists x\in X, a\oplus e=x\oplus b\] is transitive if and only if $X$ is $e$-expansive.
\end{corollary}
\begin{proof}
If $X$ is $e$-expansive then in particular it has an internal monoid structure with $x*y=2_e(x\oplus y)_X$. Hence the relation is transitive (Proposition \ref{prop R rel items}(iii)). Conversely, if the relation is transitive, then for all $x,y\in X$ the element $2_e(x\oplus y)\in A$ is in fact an element in $X$. Indeed, $2_e(x\oplus y)Ry$ and $yRe$ implies $2_e(x\oplus y)Re$ which is equivalent to $2_e(x\oplus y)\in X$. This shows that $X$ is $e$-expansive.
\end{proof}

Combining the two previous results on symmetry and transitivity with $X$ being $e$-expansive and $e$-symmetric we also observe:

\begin{corollary}\label{corollary3} Let $(A,\oplus)$ be a homogeneous ccm-magma with a subalgebra $X\subseteq A$ and an idempotent element $e\in X$. The relation \[aRb\Leftrightarrow \exists x\in X, a\oplus e=x\oplus b\] is a congruence, if and only if $X$ is homogeneous.
\end{corollary}
\begin{proof}

If $X$ is homogeneous then it is $e$-expansive with $2_e(x)\in X$ and it is $e$-symmetric with $-_e(x)=2_x(e)\in X$, for every $x\in X$. As a consequence the relation $R$ is transitive and symmetric, and hence it is a congruence, since it is always an internal reflexive relation.
Conversely, let us suppose $R$ is a congruence. By Corollary \ref{corollary2} and Proposition \ref{prop R rel items}(iv) we already know that $X$ is $e$-expansive and $e$-symmetric, hence the result in Proposition \ref{prop_homogeneous} concludes the proof.
%
%
\end{proof}

\section{Conclusion}

This work shows that the category of ccm-magmas admits several classifications for its objects. One possibility is to differentiate between those ccm-magmas admitting an internal monoid structure and those who don't. A ccm-magma $(A,\oplus)$ with a given idempotent, $e$, admits an internal monoid structure with $e$ as its unit if and only if the equation \[x\oplus e=a\oplus b\] has a solution $x$ for every  $a,b$ in $A$. This condition is weaker than $A$ being $e$-expansive. However, the two conditions are equivalent when every element $a\in A$ can be decomposed as ${a=x_1\oplus x_2}$. This property is considered, for instance, in \cite{Jezek93}.


It is shown that every relation $R$ of the form (\ref{eq_strongrel}), constructed with the two homomorphisms $f$ and $g$, is necessarily a difunctional relation. This result is also a consequence of the fact that the relation $R$ is a strong relation and  the category of ccm-magmas is weakly Mal'tsev. More generally, when $f$ and/or $g$ are not homomorphisms, it might happen that $R$ is still  an  internal relation but not  a difunctional one. In a similar way,  if $f$ and $g$ are as in Proposition \ref{prop_intrel1}, with $f(a,a)=g(a,a)$, but not homomorphisms,  then we may have a  reflexive internal relation which is not a congruence. For example, the relation $R$ in Proposition \ref{prop R rel items}, is equivalently defined as $aRb$ if and only if $f(a,b)=g(a,b)$, where, for all $a,b\in A$, $f(a,b)=a\oplus e$ while $g(a,b)=a\oplus e$ if there exists $x\in X$ such that $a\oplus e=x\oplus b$, otherwise $g(a,b)=e\oplus b$.


Every finite ccm-magma is necessarily homogeneous (Definition \ref{def_homogeneous}), and since axiom (M3) is weaker than associativity, these kinds of structures may be useful to the random  generation of finite abelian groups. The procedure is very simple: randomly generate a ccm-magma $M$ with at least one idempotent, say $e$ (although this is only important if we are interested in internal structures), and then define
\begin{center}
\texttt{A(i,j)=find(M(:,e)==M(i,j))}
\end{center} 
for every $i$ and $j$, in order to obtain a matrix $A$ with the multiplication table for an abelian group with $e$ as unit element.

The notion of ccm-magma may also be defined internally in every category with binary products (as it is done in \cite{Escardo} for midpoint algebras) and so,  some interesting interactions at this level are also expected, especially for the case of topological ccm-magmas.
%

\end{document}